\documentclass[11pt,reqno]{amsart}
\usepackage{amssymb,latexsym}
\usepackage{graphicx}

\usepackage{soul}

\usepackage{float}
\usepackage{url}
\usepackage[usenames]{color}
\usepackage[colorlinks=true, citecolor=blue, linktocpage]{hyperref}
\usepackage{cases} % Allows you to label the separate cases in the cases environment by using \begin{numcases} or \begin{subnumcases}
\usepackage{colonequals} % Allows you to use the command \colonequals to produce := in proper formatting
\usepackage[center]{caption}

\allowdisplaybreaks
\definecolor{shadecolor}{gray}{0.90}				
\def\boitegrise#1#2{\begin{centerline}{\fcolorbox{black}{shadecolor}{~
    \begin{minipage}[t]{#2}{\vphantom{~}#1\vphantom{$A_{\displaystyle{A_A}}$}}
            \end{minipage}~}}\end{centerline}\medskip}

\calclayout

\makeatletter
\newcommand{\monthyear}[1]{%
  \def\@monthyear{\uppercase{#1}}}
\newcommand{\volnumber}[1]{%
  \def\@volnumber{\uppercase{#1}}}
\AtBeginDocument{%
\def\ps@plain{\ps@empty
  \def\@oddfoot{\@monthyear \hfil \thepage}%
  \def\@evenfoot{\thepage \hfil \@volnumber}}
\def\ps@firstpage{\ps@plain}
\def\ps@headings{\ps@empty
  \def\@evenhead{%
    \setTrue{runhead}%
    \def\thanks{\protect\thanks@warning}%
    \uppercase{The Fibonacci Quarterly}\hfil}%
  \def\@oddhead{%
    \setTrue{runhead}%
    \def\thanks{\protect\thanks@warning}%
    \hfill\uppercase{GCD of sums of consecutive squares of Gibonacci numbers}}%
  \let\@mkboth\markboth
  \def\@evenfoot{%
    \thepage \hfil \@volnumber}%
  \def\@oddfoot{%
    \@monthyear \hfil \thepage}%
  }%
\footskip=25pt
\pagestyle{headings}%
}
\makeatother

\newcommand{\pmodd}[1]{\!\!\!\!\pmod{#1}}
\newcommand{\FibSum}{\mathcal{F}(k)}
\newcommand{\LucSum}{\mathcal{L}(k)}
\newcommand{\GibSum}{\mathcal{G}_{G_0, G_1}\!(k)}
\newcommand{\GibSumPrime}{\mathcal{G}_{G_0^\prime, G_1^\prime}\!(k)}
\newcommand{\GibSumPrimeSquared}{\mathcal{G}_{G_0^\prime, G_1^\prime}^2\!(k)}
\newcommand{\Gisano}{\pi_{G_0,G_1}\!(m)}

\newcommand{\GibSeq}{\left(G_n\right)_{n \geq 0}}

\newcommand{\Dfunc}{D_{G_n,G_{n+1}}}
\newcommand{\Dinitial}{D_{G_0,G_1}}
\newcommand{\GibSumSquared}{\mathcal{G}_{G_0, G_1}^2\!(k)}
\newcommand{\LucSumSquared}{\mathcal{L}^2(k)}
\newcommand{\FibSumSquared}{\mathcal{F}^2(k)}
\newcommand{\red}[1]{\textcolor{red}{#1}}

\theoremstyle{plain}
\numberwithin{equation}{section}
\newtheorem{thm}{Theorem}[section]
\newtheorem{theorem}[thm]{Theorem}
\newtheorem{lemma}[thm]{Lemma}
\newtheorem{example}[thm]{Example}
\newtheorem{definition}[thm]{Definition}

\newtheorem{corollary}[thm]{Corollary}

\theoremstyle{definition}
\newtheorem{question}[thm]{Question}
\newtheorem{convention}[thm]{Convention}

\theoremstyle{remark}
\newtheorem{remark}[thm]{Remark}

\begin{document}
%% replace the values in the next three lines by the correct information
\monthyear{Month Year}
\volnumber{Volume, Number}
\setcounter{page}{1}

\title{GCD of sums of $k$ consecutive squares of generalized Fibonacci numbers}
\author{aBa Mbirika}
\address{Department of Mathematics\\
University of Wisconsin-Eau Claire\\
Eau Claire, WI 54702\\
USA}
\email{\textcolor{blue}{mbirika@uwec.edu}}

\author{J\"urgen Spilker}
\address{Mathematisches Institut\\
Albert-Ludwigs-Universit\"at\\
D-79104 Freiburg\\
Germany}
\email{\textcolor{blue}{juergen.spilker@t-online.de}}

\begin{abstract}
In 2021, Guyer and Mbirika gave two equivalent formulas that computed the greatest common divisor (GCD) of all sums of $k$ consecutive terms in the generalized Fibonacci sequence $\left(G_n\right)_{n \geq 0}$ given by the recurrence $G_n = G_{n-1} + G_{n-2}$ for all $n \geq 2$ with integral initial conditions $G_0$ and $G_1$. In this current paper, we extend their results to the GCD of all sums of $k$ consecutive squares of these numbers. Denoting these GCD values by the symbol $\mathcal{G}_{G_0, G_1}^2\!(k)$, we prove $\mathcal{G}_{G_0, G_1}^2\!(k) = \gcd\left(G_k G_{k+1} - G_0 G_1,\; G_{k+1}^2 - G_1^2,\; G_{k+2}^2 - G_2^2\right)$. Moreover, we provide very tantalizing closed forms in the specific settings of the Fibonacci, Lucas, and generalized Fibonacci numbers. We close with a number of open questions for further research.
\end{abstract}

\maketitle

%\vspace{-.3in}

%\tableofcontents

%%%%%%%%%%%%%%%%%%%%%%%%%%%%%%%%%%%%%
%%%%%%%%%%%%%%%%%%%%%%%%%%%%%%%%%%%%%
%%%%%%%%%    SECTION 1    %%%%%%%%%%%
%%%%%%%%%%%%%%%%%%%%%%%%%%%%%%%%%%%%%
%%%%%%%%%%%%%%%%%%%%%%%%%%%%%%%%%%%%%

\section{Introduction and motivation}\label{sec:introduction}

Appearing in the literature as early as 1901 by Tagiuri~\cite{Tagiuri1901}, the generalized Fibonacci numbers (or so-called \textit{Gibonacci numbers}\footnote{Thomas Koshy attributes Art Benjamin and Jennifer Quinn for coining this term ``Gibonacci'' in their 2003 book \textit{Proofs that Really Count: The Art of Combinatorial Proof}~\cite{Benjamin2003}.}) are defined by the recurrence
$$G_n=G_{n-1}+G_{n-2} \ \text{for all} \ n\geq 2$$
with initial conditions $G_0,G_1 \in \mathbb{Z}$. In 1963, a problem was proposed by I.~D.~Ruggles in the inaugural issue of the \textit{Fibonacci Quarterly} on a closed form for the sum of any twenty consecutive Fibonacci numbers~\cite{Ruggles1963}. Since then, there has been numerous papers exploring the sums of consecutive Fibonacci or Lucas numbers~\cite{Iyer1969, Tekcan2007, Tekcan2008, Demirturk2010, Cerin2013, Shtefan2018}. This current paper continues this line of research, extending it to the greatest common divisor (GCD) of sums of certain powers of Gibonacci numbers. This current work extends earlier work of Guyer and Mbirika who explored the GCD of the sums of $k$ consecutive Gibonacci numbers, and consequently $k$ consecutive Fibonacci and Lucas numbers~\cite{Guyer_Mbirika2021}. More precisely, given $k \in \mathbb{N}$ they found the exact value of the GCD of an infinite number of finite sums $\sum_{i=1}^k G_i \;,\;\; \sum_{i=2}^{k+1} G_i \;,\;\; \sum_{i=3}^{k+2} G_i \;, \ldots$ thereby computing the GCD of the terms in the sequence $( \sum_{i=1}^k G_{n+i})_{n \geq 0}$. For brevity, they used the symbols $\FibSum$, $\LucSum$, and $\GibSum$, respectively, to denote the three values
$$ \gcd\left\lbrace \left( \sum_{i=1}^kF_{n+i} \right)_{n \geq 0}\right\rbrace, \; \gcd\left\lbrace \left( \sum_{i=1}^kL_{n+i} \right)_{n \geq 0}\right\rbrace, \text{ and } \gcd\left\lbrace \left( \sum_{i=1}^kG_{n+i} \right)_{n \geq 0}\right\rbrace,$$
where $(F_n)_{n\geq 0}$, $(L_n)_{n\geq 0}$, and $(G_n)_{n\geq 0}$ denote the Fibonacci, Lucas, and Gibonacci sequences. A main result of Guyer and Mbirika was the formula $\GibSum = \gcd(G_{k+1}-G_1,\, G_{k+2}-G_2)$~\cite[Theorem~15]{Guyer_Mbirika2021}, or an equivalent formula using generalized Pisano periods~\cite[Theorem~25]{Guyer_Mbirika2021}. Using either formula yields the values of $\FibSum$, $\LucSum$, and $\GibSum$ in Table~\ref{table:guyer_aBa_results}.

\vspace{-.165in}

\begin{table}[H]
\renewcommand{\arraystretch}{1}
\begin{center}
 \begin{tabular}{|c||c|c|c|c|}
 \hline
 $k$ & \;$\FibSum$\; & \;$\LucSum$\; & $\GibSum$\\ 
 \hline\hline
 $k\equiv 0,4,8 \pmod{12}$ & $F_{k/2}$ & $5F_{k/2}$ & $F_{k/2}^{\,\red a}$ \text{or} $5F_{k/2}^{\,\red b}$\\ 
 \hline
 $k\equiv 2,6,10 \pmod{12}$ & $L_{k/2}$ & $L_{k/2}$ & $L_{k/2}$\\
 \hline
 $k\equiv 3,9 \pmod{12}$ & 2 & 2 & $2^{\,\red c}$\\
 \hline
 $k\equiv 1,5,7,11 \pmod{12}$ & 1 & 1 & $1^{\,\red c}$\\
 \hline
\end{tabular}
\caption{Closed forms for the values $\FibSum$, $\LucSum$, and $\GibSum$}
\label{table:guyer_aBa_results}
\footnotesize{\red{$^a$} This value holds if and only if $\gcd(G_0 + G_2, G_1 + G_3) = 1$.\\
\red{$^b$} This value holds if and only if $\gcd(G_0 + G_2, G_1 + G_3) \neq 1$.\\
\red{$^c$}  These values hold if $G_1^2 - G_0 G_1 - G_0^2 = \pm 1$.} 
\end{center}
\end{table}
\vspace{-.165in}

In this current paper, we explore the GCD of sums of \textbf{squares} of $k$ consecutive generalized Fibonacci numbers (and in particular, the Fibonacci and Lucas numbers).  For brevity, we use the symbols $\FibSumSquared$, $\LucSumSquared$, and $\GibSumSquared$, respectively, to denote the three values
$$ \gcd\left\lbrace \left( \sum_{i=1}^kF_{n+i}^2 \right)_{n \geq 0}\right\rbrace, \; \gcd\left\lbrace \left( \sum_{i=1}^kL_{n+i}^2 \right)_{n \geq 0}\right\rbrace, \text{ and } \gcd\left\lbrace \left( \sum_{i=1}^kG_{n+i}^2 \right)_{n \geq 0}\right\rbrace.$$

This current paper arose when second author Spilker offered a proof of a conjecture, stating $\FibSumSquared = F_k$ when $k$ is even, given in the paper of Guyer and Mbirika~\cite[Question~54]{Guyer_Mbirika2021}, who wrote ``We feel this is simply too beautiful a result to not be true.'' Spilker's proof of this conjecture is a motivation for this current paper, wherein we extend the latter Fibonacci result to the more general setting of $\GibSumSquared$ for all $k$ values. In this paper, we prove the following main results given in Table~\ref{table:aBa_Juergen_results}. Note that the value $\mu$ (given in Definition~\ref{def:characteristic_mu_parameter}) equals $G_1^2 - G_0 G_1 - G_0^2$, and the value $g_k$ equals $\gcd\left(G_{k+1}^2 - G_1^2,\; G_{k+2}^2 - G_2^2\right)$.

\begin{table}[H]
\renewcommand{\arraystretch}{1.1}
\begin{center}
    \begin{tabular}{|c||c|c|c|c|}
    \hline
         $k$ & $\FibSumSquared$ & $\LucSumSquared$ & $\GibSumSquared$ & Proof in this paper \\ \hline\hline
         $k$ even and $5\nmid\mu$ & $F_k$ & $5F_k$ & $F_k$ & Theorems~\ref{thm:FibSumSquared_closed_forms}, \ref{thm:LucSumSquared_closed_forms}, and \ref{thm:closed_form_for_GibSumSquared_when_k_even}, respectively\\ \hline
         $k$ even and $5\mid\mu$ & $F_k$ & $5F_k$ & $5F_k$ & Theorems~\ref{thm:FibSumSquared_closed_forms}, \ref{thm:LucSumSquared_closed_forms}, and \ref{thm:closed_form_for_GibSumSquared_when_k_even}, respectively\\ \hline
         $k \equiv 3 \pmod{6}$ & 2 & 2 & $\gcd(2\mu, g_k)$ & Theorems~\ref{thm:FibSumSquared_closed_forms}, \ref{thm:LucSumSquared_closed_forms}, and \ref{thm:Simpler_GibSum_formula_for_k_even_versus_odd}, respectively\\ \hline
         $k \equiv 1,5 \pmod{6}$ & 1 & 1 & $\gcd(2\mu, g_k)$ & Theorems~\ref{thm:FibSumSquared_closed_forms}, \ref{thm:LucSumSquared_closed_forms}, and \ref{thm:Simpler_GibSum_formula_for_k_even_versus_odd}, respectively\\ \hline
    \end{tabular}
\caption{Closed forms for the values $\FibSumSquared$, $\LucSumSquared$, and $\GibSumSquared$}
\label{table:aBa_Juergen_results}
\end{center}
\end{table}

\begin{remark}
Compare Tables~\ref{table:guyer_aBa_results} and $\ref{table:aBa_Juergen_results}$ and the closed forms for the values $\FibSum$ and $\FibSumSquared$, respectively, when $k$ is even. Observe that $\FibSum = F_{k/2}$ when $k \equiv 0 \pmod{4}$, whereas $\FibSum = L_{k/2}$ when $k \equiv 2 \pmod{4}$. However in this new setting of the GCD of sums of squares of $k$ consecutive Fibonacci numbers, the values $\FibSumSquared$ have no dependency on the residue class modulo 4 of $k$ when $k$ is even, so in some sense $\FibSumSquared$ is more well behaved than its seemingly simpler counterpart $\FibSum$. Similar phenomena happens for $k$ even when we compare $\LucSum$ and $\LucSumSquared$.
\end{remark}

The paper is broken down as follows. In Section~\ref{sec:definitions}, we provide the necessary definitions used in the subsequent sections. In Section~\ref{sec:generalized_Fibonacci_setting}, we give a simple formula for the value $\GibSumSquared$. In Section~\ref{sec:closed_form_in_Gibonacci_setting}, we provide nice closed forms for the value $\GibSumSquared$. Moreover in Section~\ref{sec:closed_form_in_Fibonacci_and_Lucas_settings}, we provide similarly nice closed forms for the values $\FibSumSquared$ and $\LucSumSquared$, respectively. Finally in Section~\ref{sec:open questions}, we end with some tantalizing open questions.

%%%%%%%%%%%%%%%%%%%%%%%%%%%%%%%%%%%%%
%%%%%%%%%%%%%%%%%%%%%%%%%%%%%%%%%%%%%
%%%%%%%%%    SECTION 2    %%%%%%%%%%%
%%%%%%%%%%%%%%%%%%%%%%%%%%%%%%%%%%%%%
%%%%%%%%%%%%%%%%%%%%%%%%%%%%%%%%%%%%%

\section{Preliminary definitions}\label{sec:definitions}

\begin{definition}
The \textit{generalized Fibonacci sequence} $\left(G_n\right)_{n \geq 0}$ is defined by the recurrence relation
\begin{equation*}
G_n = G_{n-1} + G_{n-2}
\end{equation*}
for $n \geq 2$ and arbitrary initial conditions $G_0,G_1 \in \mathbb{Z}$. The \textit{Fibonacci sequence} $\left(F_n\right)_{n \geq 0}$ is recovered when $G_0=0$ and $G_1=1$, and the \textit{Lucas sequence} $\left(L_n\right)_{n \geq 0}$ is recovered when $G_0=2$ and $G_1=1$. For brevity, we use the term \textit{Gibonacci sequence} to refer to any generalized Fibonacci sequence.
\end{definition}

\boitegrise{
\begin{convention}\label{conv:relatively_prime_initial_values}
For reasons explained in Theorem~\ref{thm:relatively_prime_initial_values_only} and Convention~\ref{conv:use_only_relatively_prime_initial_values}, it suffices to consider only the Gibonacci sequences with relatively prime initial conditions $G_0$ and $G_1$.
\vspace{-.2in}
\end{convention}}{0.97\textwidth}

\begin{definition}\label{def:characteristic_mu_parameter}
The \textit{characteristic} of the Gibonacci sequence $(G_n)_{n \geq 0}$ is denoted $\mu$ and is defined as $\mu = G_1^2 - G_0 G_1 - G_0^2$.
\end{definition}

\begin{remark}\label{rem:differing_mu_notations}
Many authors have differing notation and/or equivalent definitions for this value $\mu$. For instance, Guyer-Mbirika denote this value as $\Dinitial$ to highlight this value's dependency on the initial conditions~\cite{Guyer_Mbirika2021}. Moreover, Koshy maintains the symbol $\mu$, but defines it as $G_1^2 + G_1 G_2 - G_2^2$~\cite{Koshy2018}. On the other hand, Vajda reserves no symbol for $\mu$ but writes the value as $G_1^2 - G_0 G_2$~\cite{Vajda1989}.
\end{remark}

Lastly, it is well known that the Fibonacci and Lucas sequences (and more generally any Gibonacci sequence) under a modulus are periodic. The lengths of the periods of these sequences are called Pisano periods. More generally, we have the following definition of the length of the period of a Gibonacci sequence under a modulus.

\begin{definition}\label{def:Gisano_period}
Let $m \geq 2$. The \textit{generalized Pisano period}, $\Gisano$, of the Gibonacci sequence $\GibSeq$ is the smallest positive integer $r$ such that
$$ G_r \equiv G_0 \pmodd{m} \;\;\text{ and }\;\; G_{r+1} \equiv G_1 \pmodd{m}.$$
The value $r$ is dependent on both the initial conditions $G_0,G_1 \in \mathbb{Z}$ and the modulus $m$. In the Fibonacci (respectively, Lucas) setting we denote this period by $\pi_F(m)$ (respectively, $\pi_L(m)$).
\end{definition}

%%%%%%%%%%%%%%%%%%%%%%%%%%%%%%%%%%%%%
%%%%%%%%%%%%%%%%%%%%%%%%%%%%%%%%%%%%%
%%%%%%%%%    SECTION 3    %%%%%%%%%%%
%%%%%%%%%%%%%%%%%%%%%%%%%%%%%%%%%%%%%
%%%%%%%%%%%%%%%%%%%%%%%%%%%%%%%%%%%%%

\section{The generalized Fibonacci setting}\label{sec:generalized_Fibonacci_setting}

\subsection{A formula for \texorpdfstring{$\GibSumSquared$}{curly G squared}}\label{subsec:simple_formulas_for_curly_G}
In this subsection, we derive the following formula for $\GibSumSquared$ in Theorem~\ref{thm:Simple_GibSum_formula}:
$$ \GibSumSquared = \gcd\left(G_k G_{k+1} - G_0 G_1,\; G_{k+1}^2 - G_1^2,\; G_{k+2}^2 - G_2^2\right).$$

\begin{lemma}\label{lem:gcd_of_infinite_sequence}
Let $(a_n)_{n\geq 0}$ be a sequence of integers. The following identity holds:
$$ \gcd(a_0, a_1, a_2, a_3, \ldots) = \gcd(a_0,\; a_1 - a_0,\; a_2 - a_1,\; a_3 - a_2, \ldots). $$
\end{lemma}

\begin{proof}
Set $d\colonequals \gcd(a_0, a_1, a_2, a_3, \ldots)$ and $e\colonequals \gcd(a_0,\; a_1 - a_0,\; a_2 - a_1,\; a_3 - a_2, \ldots)$. We will show that $d = e$. By assumption, $d$ divides $a_i$ for all $i \geq 0$, and hence $d$ divides $a_{i+1} - a_i$ for all $i \geq 0$. Hence $d$ divides $e$ and so $d \leq e$. To prove the reverse inequality, observe that $e$ divides both $a_0$ and $a_1 - a_0$, and hence $e$ divides $a_1$. Consequently since $e$ divides $a_1$ and $a_2 - a_1$, then $e$ divides $a_2$. Continuing inductively in this manner, we see that $e$ must divide $a_i$ for all $i$. Hence $e$ divides $d$ and so $e \leq d$.
\end{proof}

The next two lemmas are well-known results whose proofs can be found in Vajda~\cite{Vajda1989}.

\begin{lemma}\label{lem:sum_of_first_Gib_numbers_squared}
For all $k \geq 1$, we have
$$ \sum_{i=1}^k G_i^2 = G_k G_{k+1} - G_0 G_1. $$
\end{lemma}

\begin{proof}
See Identity~(44) of Vajda~\cite[p.~43]{Vajda1989}.
\end{proof}

\begin{lemma}\label{lem:Gibonacci_Vorobiev_identity}
For all $m,n \in \mathbb{Z}$, we have
$$ G_{m+n} = G_{m+1} F_n + G_m F_{n-1}. $$
\end{lemma}

\begin{proof}
See Identity~(8) of Vajda~\cite[pp.~24--25]{Vajda1989}.
\end{proof}

We are now ready to prove the main theorem of this subsection, namely a formula for the value $\GibSumSquared$.

\begin{theorem}\label{thm:Simple_GibSum_formula}
For all $k \geq 1$, we have
$$ \GibSumSquared = \gcd\left(G_k G_{k+1} - G_0 G_1,\; G_{k+1}^2 - G_1^2,\; G_{k+2}^2 - G_2^2\right).$$
\end{theorem}

\begin{proof}
Fix $k \geq 1$ and for $n \geq 0$, set $S_n \colonequals \sum_{i=1}^k G_{n+i}^2$ and $H_n \colonequals S_{n+1} - S_n$. Hence we have
\begin{align*}
    \GibSumSquared &= \gcd \left( S_0, S_1, S_2, S_3, \ldots \right) = \gcd \left(S_0, H_0, H_1, H_2, \ldots\right),
\end{align*}
where the second equality holds by Lemma~\ref{lem:gcd_of_infinite_sequence}. We will show that for all $n \geq 2$, the value $H_n$ is a linear combination of $S_0$, $H_0$, and $H_1$, and hence the last equality above coincides with $\gcd(S_0, H_0, H_1)$, and the result follows. To that end, we first prove the following two equivalent expressions for $S_n$ and $H_n$, respectively:
\begin{align}
    S_n &= G_{n+k} G_{n+k+1} - G_n G_{n+1} \label{eq:S_n_formula}\\
    H_n &= G_{n+(k+1)}^2 - G_{n+1}^2. \label{eq:H_n_formula}
\end{align}
Identity~\eqref{eq:S_n_formula} holds by the sequence of equalities
\begin{align*}
    S_n &= \sum_{i=1}^k G_{n+i}^2\\
    &= \sum_{i=1}^{n+k} G_i^2 - \sum_{i=1}^{n} G_i^2\\
    &= \left( G_{n+k} G_{n+k+1} - G_0 G_1 \right) - \left( G_{n} G_{n+1} - G_0 G_1 \right) &\text{by Lemma~\ref{lem:sum_of_first_Gib_numbers_squared}}\\
    &= G_{n+k} G_{n+k+1} - G_n G_{n+1}.
\end{align*}
Moreover, Identity~\eqref{eq:H_n_formula} holds by the sequence of equalities
\begin{align*}
    H_n &= S_{n+1} - S_n\\
    &= \sum_{i=1}^k G_{n+1+i}^2 - \sum_{i=1}^k G_{n+i}^2\\
    &= ( G_{n+2}^2 + G_{n+3}^2 + \cdots + G_{n+(k+1)}^2 ) - ( G_{n+1}^2 + G_{n+2}^2 + \cdots + G_{n+k}^2 )\\
    &= G_{n+(k+1)}^2 - G_{n+1}^2.
\end{align*}
Applying Lemma~\ref{lem:Gibonacci_Vorobiev_identity} to both $G_{n+(k+1)}$ and $G_{n+1}$ in the latter equality and expanding the squared terms, we have
\begin{align*}
    H_n &= \left(G_{k+2} F_n + G_{k+1} F_{n-1} \right)^2 - \left(G_2 F_n + G_1 F_{n-1} \right)^2\\
    &= \left( G_{k+2}^2 F_n^2 + G_{k+1}^2 F_{n-1}^2 +2 (G_{k+1} G_{k+2}) (F_{n-1} F_n) \right)\\ 
    & \hspace{2.25in} - \left( G_2^2 F_n^2 + G_1^2 F_{n-1}^2 +2 (G_1 G_2) (F_{n-1} F_n) \right)\\
    &= \left( G_{k+2}^2 - G_2^2 \right) \cdot F_n^2 + \left( G_{k+1}^2 - G_1^2 \right) \cdot F_{n-1}^2\\
    & \hspace{2.25in} + 2 \left( G_{k+1} G_{k+2} - G_1 G_2 \right) \cdot (F_{n-1} F_n)\\
    &= H_1 F_n^2 + H_0 F_{n-1}^2 + 2 S_1 (F_{n-1} F_n).
\end{align*}
where the last equality holds by Identities~\eqref{eq:S_n_formula} and \eqref{eq:H_n_formula}. Lastly, since $H_0 = S_1 - S_0$ implies $S_1 = H_0 + S_0$, we can rewrite the last equality as
$$ H_n = S_0 (2 F_{n-1} F_n) + H_0 (F_{n-1}^2 + 2 F_{n-1} F_n) + H_1 F_n^2, $$
and hence for $n \geq 2$, we see that $H_n$ is an integer linear combination of $S_0$, $H_0$, and $H_1$, as desired. We conclude that
\begin{align*}
    \GibSumSquared &= \gcd \left(S_0, H_0, H_1, H_2, \ldots\right)\\
    &= \gcd \left(S_0, H_0, H_1\right)\\
    &= \gcd\left(G_k G_{k+1} - G_0 G_1,\; G_{k+1}^2 - G_1^2,\; G_{k+2}^2 - G_2^2\right),
\end{align*}
where the last equality holds by Identities~\eqref{eq:S_n_formula} and \eqref{eq:H_n_formula}.
\end{proof}

Next we show that it is sufficient to explore only the Gibonacci sequences which have relatively prime initial values

\begin{lemma}\label{lem:Consec_Gib_GCD}
For all $n \in \mathbb{Z}$, the values $\gcd(G_{n+1}, G_{n+2})$ and $\gcd(G_{n}, G_{n+1})$ coincide. In particular, $\gcd(G_0, G_1) = \gcd(G_n, G_{n+1})$ for all $n \in \mathbb{Z}$.
\end{lemma}

\begin{proof}
See Lemma~18 of Guyer-Mbirika~\cite{Guyer_Mbirika2021}.
\end{proof}

\begin{theorem}\label{thm:relatively_prime_initial_values_only}
Fix $G_0, G_1 \in \mathbb{Z}$ and set $d := \gcd(G_0,G_1)$. Consider the two Gibonacci sequences $(G_n)_{n\geq 0}$ and $(G_n^\prime)_{n\geq 0}$, where $(G_n^\prime)_{n=0}^\infty$ is generated by the relatively prime initial conditions $G_0^\prime = \frac{G_0}{d}$ and $G_1^\prime = \frac{G_1}{d}$. Then the following equality holds:
$$ \GibSumSquared = d^2 \cdot \GibSumPrimeSquared.$$
\end{theorem}

\begin{proof}
Set $d \colonequals \gcd(G_0,G_1)$. By Lemma~\ref{lem:Consec_Gib_GCD}, we have $\gcd(G_{k+1},G_{k+2}) = \gcd(G_0,G_1) = d$ for all $k \in \mathbb{Z}$. By Theorem~\ref{thm:Simple_GibSum_formula}, we have $\GibSumSquared = \gcd\left(G_k G_{k+1} - G_0 G_1,\; G_{k+1}^2 - G_1^2,\; G_{k+2}^2 - G_2^2\right)$. Moreover, since $d$ divides $G_0$ and $G_1$, then $d$ divides every term in the sequence $\GibSeq$. In particular, $\frac{G_k G_{k+1} - G_0 G_1}{d^2}$, $\frac{G_{k+1}^2 - G_1^2}{d^2}$, and $\frac{G_{k+2}^2 - G_2^2}{d^2}$ are integers. Observe the sequence of equalities
\begin{align*}
    \GibSumSquared &= \gcd\left(G_k G_{k+1} - G_0 G_1,\; G_{k+1}^2 - G_1^2,\; G_{k+2}^2 - G_2^2\right)\\
    &=\gcd\left(d^2 \cdot  \frac{G_k G_{k+1} - G_0 G_1}{d^2},\; d^2 \cdot \frac{G_{k+1}^2 - G_1^2}{d^2},\; d^2 \cdot \frac{G_{k+2}^2 - G_2^2}{d^2}\right)\\
    &= d^2 \cdot \gcd\left(\frac{G_k G_{k+1} - G_0 G_1}{d^2},\; \frac{G_{k+1}^2 - G_1^2}{d^2},\; \frac{G_{k+2}^2 - G_2^2}{d^2}\right).
\end{align*}
However, by Theorem~\ref{thm:Simple_GibSum_formula}, the value $\gcd\left(\frac{G_k G_{k+1} - G_0 G_1}{d^2},\; \frac{G_{k+1}^2 - G_1^2}{d^2},\; \frac{G_{k+2}^2 - G_2^2}{d^2}\right)$ is the GCD of the sum of $k$ consecutive squares of Gibonacci numbers in the new sequence $(G_n^\prime)_{n=0}^\infty$ generated by the initial values $G_0^\prime = \frac{G_0}{d}$ and $G_1^\prime = \frac{G_1}{d}$. Clearly $G_0^\prime$ and $G_1^\prime$ are relatively prime. In particular, we have $\GibSumSquared = d^2 \cdot \GibSumPrimeSquared$, as desired.
\end{proof}

\boitegrise{
\begin{convention}\label{conv:use_only_relatively_prime_initial_values}
In order to give a complete classification of the GCD of every sum of $k$ consecutive squares of Gibonacci numbers, as a consequence of Theorem~\ref{thm:relatively_prime_initial_values_only}, we need only to consider Gibonacci sequences with relatively prime initial values.
\vspace{-.2in}
\end{convention}}{0.9\textwidth}

%%%%%%%%%%%%%%%%%%%%%%%%%%%%%%%
%%%%%%%%%%%%%%%%%%%%%%%%%%%%%%%

\subsection{Simplified formulas for \texorpdfstring{$\GibSumSquared$}{curly G squared} when \texorpdfstring{$k$}{k} is even versus odd}\label{subsec:closed_form_for_curly_G}

In this subsection, we reveal that the formula for $\GibSumSquared$ given in Theorem~\ref{thm:Simple_GibSum_formula} in the previous subsection can be simplified further if we consider the two cases of the parity of $k$ as follows:
$$ \GibSumSquared = \begin{cases}
  \gcd\left(G_{k+1}^2 - G_1^2,\; G_{k+2}^2 - G_2^2\right),  & \text{if $k$ is even}, \\
  \gcd\left(2 \mu,\; G_{k+1}^2 - G_1^2,\; G_{k+2}^2 - G_2^2\right), & \text{ if $k$ is odd},
\end{cases} $$
where the characteristic $\mu$ is defined as $G_1^2 - G_0 G_1 - G_0^2$ (recall Definition~\ref{def:characteristic_mu_parameter} and Remark~\ref{rem:differing_mu_notations}). Pivotal in the proof of the formulas above is the use of the following lemma known as Cassini's identity for generalized Fibonacci sequences.

\begin{lemma}[Generalized Cassini's Identity]\label{lem:generalized_Cassini} 
Let $n \geq 1$ be given. Then the following identity holds: $G_{n+1} G_{n-1} - G_n^2 = (-1)^n \cdot \mu$, where $\mu = G_1^2 - G_0 G_1 - G_0^2$. 
\end{lemma}
\begin{proof}
See Identity~(28) of Vajda~\cite[p.~32]{Vajda1989}.
\end{proof}

\begin{lemma}\label{lem:D_function_invariance}
Let $\Dfunc$ denote the value $G_{n+1}^2 - G_n G_{n+1} - G_n^2$. Then the following holds:
\begin{align}
    \Dfunc = (-1)^n \cdot \mu \label{eq:Dfunc_invariance}
\end{align}
for all $n \geq 0$, where $\mu = \Dinitial$ (i.e., the characteristic $G_1^2 - G_0 G_1 - G_0^2$). In particular, we have $|\Dfunc| = |\mu|$ for all $n \geq 0$.
\end{lemma}

\begin{proof}
See Lemma~34 of Guyer-Mbirika~\cite{Guyer_Mbirika2021}.
\end{proof}

\begin{lemma}\label{lem:zero_versus_4mu}
For all $k \geq 1$, we have
$$ \left(G_k^2 - G_0^2\right) - 3\left(G_{k+1}^2 - G_1^2\right) + \left(G_{k+2}^2 - G_2^2\right) = \begin{cases}
  0,  & \text{if $k$ is even}, \\
  4 \mu, & \text{ if $k$ is odd},
\end{cases}$$
where $\mu = G_1^2 - G_0 G_1 - G_0^2$.
\end{lemma}

\begin{proof}
For ease of notation, set $M_i \colonequals G_{k+i}^2 - G_i^2$. Hence it suffices to show the following:
$$ M_0 - 3 M_1 + M_2 = \begin{cases}
  0,  & \text{if $k$ is even}, \\
  4 \mu, & \text{ if $k$ is odd}.
\end{cases}$$
Since $G_k = G_{k+2} - G_{k+1}$ and $G_0 = G_2 - G_1$, we have the sequence of equalities
\begin{align*}
    M_0 &= G_k^2 - G_0^2\\
    &= \left( G_{k+2} - G_{k+1} \right)^2 - \left( G_2 - G_1 \right)^2\\
    &= \left( G_{k+2}^2 - 2 G_{k+1} G_{k+2} + G_{k+1}^2 \right) - \left( G_2^2 - 2 G_1 G_2 + G_1^2 \right)\\
    &= \left( G_{k+2}^2 - G_2^2 \right) + \left( G_{k+1}^2 - G_1^2 \right) - 2 \left( G_{k+1} G_{k+2} - G_1 G_2 \right)\\
    &= M_2 + M_1 - 2 \left( G_{k+1} G_{k+2} - G_1 G_2 \right).
\end{align*}
Thus it follows that
\begin{align}
    M_0 - 3 M_1 + M_2 &= 2 M_2 - 2 M_1 - 2 \left( G_{k+1} G_{k+2} - G_1 G_2 \right). \label{eq:zero_versus_4mu_equation}
\end{align}
Recalling from Lemma~\ref{lem:zero_versus_4mu} that $\Dfunc = G_{n+1}^2 - G_n G_{n+1} - G_n^2$, the right side of Equation~\eqref{eq:zero_versus_4mu_equation} decomposes as
\begin{align*}
     2 M_2 - &2 M_1 - 2 \left( G_{k+1} G_{k+2} - G_1 G_2 \right)\\
     &= 2 \left( G_{k+2}^2 - G_2^2 \right) - 2 \left( G_{k+1}^2 - G_1^2 \right) - 2 \left( G_{k+1} G_{k+2} - G_1 G_2 \right)\\
     &= 2 \left( G_{k+2}^2 - G_{k+1} G_{k+2} -  G_{k+1}^2 \right) - 2 \left( G_2^2 - G_1 G_2 - G_1^2 \right)\\
     &= 2 D_{G_{k+1},G_{k+2}} - 2 D_{G_1,G_2}\\
     &= 2 \left( (-1)^{k+1} \cdot \mu \right) - 2 \left( (-1)^1 \cdot \mu \right) &\text{by Lemma~\ref{lem:zero_versus_4mu}}\\
     &= 2 \left( (-1)^{k+1} + 1 \right)\cdot \mu\\
     &= \begin{cases}
  0,  & \text{if $k$ is even}, \\
  4 \mu, & \text{ if $k$ is odd}.
\end{cases}
\end{align*}
Combining the latter equality with Equation~\eqref{eq:zero_versus_4mu_equation}, the result follows.
\end{proof}

We are now ready to state and prove the main theorem of this subsection.

\begin{theorem}\label{thm:Simpler_GibSum_formula_for_k_even_versus_odd}
For all $k \geq 1$, we have
$$ \GibSumSquared = \begin{cases}
  \gcd\left(G_{k+1}^2 - G_1^2,\; G_{k+2}^2 - G_2^2\right),  & \text{if $k$ is even}, \\
  \gcd\left(2 \mu,\; G_{k+1}^2 - G_1^2,\; G_{k+2}^2 - G_2^2\right), & \text{ if $k$ is odd},
\end{cases} $$
where $\mu = G_1^2 - G_0 G_1 - G_0^2$.
\end{theorem}

\begin{proof}
For ease of notation, set $M_i \colonequals G_{k+i}^2 - G_i^2$. Hence it suffices to show that if $k$ is even (respectively, odd), then $\GibSumSquared = \gcd(M_1, M_2)$ (respectively, $\GibSumSquared = \gcd(2\mu, M_1, M_2)$). To that end, observe the sequence of equalities
\begin{align*}
    G_k G_{k+1} - G_0 G_1 &= G_k (G_{k+2} - G_k) - G_0 (G_2 - G_0)\\
    &= G_k G_{k+2} - G_k^2 - G_0 G_2 + G_0^2\\
    &= (G_{k+1}^2 + (-1)^{k+1} \mu) - G_k^2 - (G_1^2 + (-1)^1 \mu) + G_0^2 &\text{by Lemma~\ref{lem:generalized_Cassini}}\\
    &= (G_{k+1}^2 - G_1^2) - (G_{k}^2 - G_0^2) + \left( (-1)^{k+1} + 1 \right)\cdot \mu\\
    &= M_1 - M_0 + \left( (-1)^{k+1} + 1 \right)\cdot \mu.
\end{align*}
So it follows that
\begin{subnumcases}{G_k G_{k+1} - G_0 G_1=}
    M_1 - M_0, & \text{if $k$ is even}, \label{eq:simpler_formula_k_even}
   \\
   M_1 - M_0 + 2\mu, & \text{if $k$ is odd}. \label{eq:simpler_formula_k_odd}
\end{subnumcases}
By Theorem~\ref{thm:Simple_GibSum_formula}, if $k$ is even, then we have
\begin{align*}
    \GibSumSquared &= \gcd\left(G_k G_{k+1} - G_0 G_1,\; G_{k+1}^2 - G_1^2,\; G_{k+2}^2 - G_2^2\right)\\
    &= \gcd\left(G_k G_{k+1} - G_0 G_1,M_1,M_2\right)\\
    &= \gcd(M_1 - M_0, M_1, M_2) & \text{by Equation~\eqref{eq:simpler_formula_k_even}}\\
    &= \gcd(M_0, M_1, M_2)\\
    &= \gcd(3 M_1 - M_2, M_1, M_2) &\text{by Lemma~\ref{lem:zero_versus_4mu}}\\
    &= \gcd(M_1, M_2),
\end{align*}
as desired. On the other hand, if $k$ is odd, then we have
\begin{align*}
    \GibSumSquared &= \gcd\left(G_k G_{k+1} - G_0 G_1,\; G_{k+1}^2 - G_1^2,\; G_{k+2}^2 - G_2^2\right)\\
    &= \gcd\left(G_k G_{k+1} - G_0 G_1,M_1,M_2\right)\\
    &= \gcd(M_1 - M_0 + 2\mu, M_1, M_2) & \text{by Equation~\eqref{eq:simpler_formula_k_odd}}\\
    &= \gcd(-M_0 + 2\mu, M_1, M_2)\\
    &= \gcd(-3 M_1 + M_2 - 2\mu, M_1, M_2) &\text{by Lemma~\ref{lem:zero_versus_4mu}}\\
    &= \gcd(-2\mu, M_1, M_2)\\
    &= \gcd(2\mu, M_1, M_2),
\end{align*}
as desired.
\end{proof}

%%%%%%%%%%%%%%%%%%%%%%%%%%%%%%%%%%%%%
%%%%%%%%%%%%%%%%%%%%%%%%%%%%%%%%%%%%%
%%%%%%%%%    SECTION 4    %%%%%%%%%%%
%%%%%%%%%%%%%%%%%%%%%%%%%%%%%%%%%%%%%
%%%%%%%%%%%%%%%%%%%%%%%%%%%%%%%%%%%%%

\section{Closed forms in the generalized Fibonacci setting}\label{sec:closed_form_in_Gibonacci_setting}

In this section, we prove the following closed forms on the GCD of the sum of $k$ consecutive squares of Gibonacci numbers, noting $g_k \colonequals \gcd\left( G_{k+1}^2 - G_1^2, \; G_{k+2}^2 - G_2^2 \right)$:
\begingroup
\renewcommand{\arraystretch}{1.1}
\begin{center}
    \begin{tabular}{|c||c|c|}
    \hline
         $k$ & $\GibSumSquared$ & Proof in this paper \\ \hline\hline
         $k$ even and $5\nmid\mu$ & $F_k$ & Theorem~\ref{thm:closed_form_for_GibSumSquared_when_k_even}\\ \hline
         $k$ even and $5\mid\mu$ & $5F_k$ & Theorem~\ref{thm:closed_form_for_GibSumSquared_when_k_even}\\ \hline
         $k \equiv 3 \pmod{6}$ & $\gcd(2\mu, g_k)$ & Theorem~\ref{thm:Simpler_GibSum_formula_for_k_even_versus_odd}\\ \hline
         $k \equiv 1,5 \pmod{6}$ & $\gcd(2\mu, g_k)$ & Theorem~\ref{thm:Simpler_GibSum_formula_for_k_even_versus_odd}\\ \hline
    \end{tabular}
\end{center}
\endgroup

%%%%%%%%%%%%%%%%%%%%%%%%%%%%%%%
%%%%%%%%%%%%%%%%%%%%%%%%%%%%%%%

\subsection{Closed form for \texorpdfstring{$\GibSumSquared$}{curly G squared} when \texorpdfstring{$k$}{k} is even}

\begin{lemma}\label{lem:Vorobiev_expression_for_G_i}
For all $i \in \mathbb{Z}$, we have
$$ G_i = G_0 F_{i-1} + G_1 F_i. $$
\end{lemma}

\begin{proof}
This follows from Lemma~\ref{lem:Gibonacci_Vorobiev_identity}, if we set $m \colonequals i$ and $n \colonequals 0$.
\end{proof}

The following well-known identity is a generalization of Cassini's Identity (recall Lemma~\ref{lem:generalized_Cassini}) in the Fibonacci setting and will be useful in proving Lemma~\ref{lem:difference_of_two_Fibonacci_squares}.

\begin{lemma}[Catalan's Identity]\label{lem:Catalan_identity}
For $n,r \in \mathbb{Z}$ with $n \geq r$, we have
$$F_n^2-F_{n-r}F_{n+r}=(-1)^{n-r}F_r^2.$$
\end{lemma}

\begin{proof}
See Theorem~5.11 of Koshy~\cite[p.~106]{Koshy2018}.
\end{proof}

The next lemma reveals the interesting fact that if $k$ is even, then $F_k$ divides the difference of the squares of any two Fibonacci numbers that are $k$ terms apart. We use this fact to prove the subsequent result, Lemma~\ref{lem:difference_of_two_Gibonacci_squares}, which states that if $k$ is even, then $F_k$ also divides the difference of the squares of any two Gibonacci numbers that are $k$ terms apart.

\begin{lemma}\label{lem:difference_of_two_Fibonacci_squares}
For $k \geq 0$ with $k$ even and $\ell \in \mathbb{Z}$, we have
$$ F_{k+\ell}^2 - F_\ell^2 = F_k F_{k+2\ell}. $$
\end{lemma}

\begin{proof}
By Catalan's Identity of Lemma~\ref{lem:Catalan_identity}, if we set $r \colonequals \ell$ and $n \colonequals k+\ell$, then it follows that
$$ F_{k+\ell}^2 - F_{(k+\ell)-\ell} F_{(k+\ell)+ \ell} = (-1)^{(k+\ell) - \ell} F_\ell^2. $$
Noting that $k$ is even, the latter equality reduces to $F_{k+\ell}^2 - F_k F_{k+2\ell} = F_\ell^2$, and the result follows.
\end{proof}

\begin{corollary}\label{cor:difference_of_two_Fibonacci_squares}
For $k \geq 0$ with $k$ even and $\ell \in \mathbb{Z}$, we have
$$ F_{k+\ell-1}^2 - F_{\ell-1}^2 = F_k F_{k+2\ell-2}. $$
\end{corollary}

\begin{proof}
This follows immediately from Lemma~\ref{lem:difference_of_two_Fibonacci_squares} if we replace $\ell$ with $\ell-1$.
\end{proof}

\begin{lemma}\label{lem:Vajda_Identity_20a}
For $k,\ell \geq 0$ with $k$ even, we have
\begin{align*}
    F_{k + \ell - 1} F_{k+\ell} - F_{\ell-1} F_\ell &= F_k F_{k + 2\ell -1}.
\end{align*}
\end{lemma}

\begin{proof}
By Identity~(20a) of Vajda~\cite[p.~28]{Vajda1989}, for all $a,b,c \geq 0$, we have the following identity:
\begin{align*}
    F_{a+b} F_{a+c} - (-1)^a F_b F_c &= F_a F_{a+b+c}
\end{align*}
Setting $a \colonequals k$, $b \colonequals \ell-1$, and $c \colonequals \ell$, and noting that $k$ is even, the result follows.
\end{proof}

\begin{lemma}\label{lem:difference_of_two_Gibonacci_squares}
For $k, \ell \geq 0$  with $k$ even, we have
$$ G_{k+\ell}^2 - G_\ell^2 = F_k \cdot \left(G_0^2 F_{k+2\ell-2} + 2 G_0 G_1 F_{k+2\ell-1} + G_1^2 F_{k+2\ell} \right). $$
\end{lemma}

\begin{proof}
For ease of notation, set $\gamma_{k,\ell} \colonequals G_0^2 F_{k+2\ell-2} + 2 G_0 G_1 F_{k+2\ell-1} + G_1^2 F_{k+2\ell}$. Observe that
\begin{align*}
  G_{k+\ell}^2 &- G_\ell^2\\
  &= \left( G_0 F_{k+\ell-1} + G_1 F_{k+\ell} \right)^2 - \left( G_0 F_{\ell-1} + G_1 F_\ell \right)^2 &\text{by Lemma~\ref{lem:Vorobiev_expression_for_G_i}}\\
  &= \left( G_0^2 F_{k+\ell-1}^2 + 2 G_0 G_1 F_{k+\ell-1} F_{k+\ell} + G_1^2 F_{k+\ell}^2 \right)\\
  & \hspace{2in} - \left( G_0^2 F_{\ell-1}^2 + 2 G_0 G_1 F_{\ell-1} F_{\ell} + G_1^2 F_{\ell}^2 \right)\\
  &= G_0^2 \left( F_{k+\ell-1}^2 - F_{\ell-1}^2\right) + G_1^2 \left( F_{k+\ell}^2 - F_{\ell}^2\right) + 2 G_0 G_1 \left( F_{k+\ell-1} F_{k + \ell} - F_{\ell-1} F_\ell \right)\\
  &= G_0^2 (F_k F_{k+2\ell-2}) + G_1^2 (F_k F_{k+2\ell}) + 2 G_0 G_1 \left( F_{k+\ell-1} F_{k + \ell} - F_{\ell-1} F_\ell \right) &\text{by Lemma~\ref{lem:difference_of_two_Fibonacci_squares}}\\
  &= G_0^2 (F_k F_{k+2\ell-2}) + G_1^2 (F_k F_{k+2\ell}) + 2 G_0 G_1 (F_k F_{k+2\ell-1}) &\text{by Lemma~\ref{lem:Vajda_Identity_20a}}\\
  &= F_k \cdot \left( G_0^2 F_{k+2\ell-2} + 2 G_0 G_1 F_{k+2\ell-1} + G_1^2 F_{k+2\ell} \right)\\
  &= F_k \cdot \gamma_{k,\ell},
\end{align*}
and the result follows.
\end{proof}

We are now ready to state and prove the main theorem of this subsection.

\begin{theorem}\label{thm:closed_form_for_GibSumSquared_when_k_even}
For all $k \geq 1$ with $k$ even, we have
$$ \GibSumSquared = \begin{cases}
  F_k,  & \text{if $5$ does not divide $\mu$}, \\
  5 F_k,  & \text{if $5$ divides $\mu$},
\end{cases}  $$
where $\mu = G_1^2 - G_0 G_1 - G_0^2$.
\end{theorem}

\begin{proof}
Suppose that $k \geq 1$ is even. Recalling that $\gamma_{k,\ell} \colonequals G_0^2 F_{k+2\ell-2} + 2 G_0 G_1 F_{k+2\ell-1} + G_1^2 F_{k+2\ell}$ from Lemma~\ref{lem:difference_of_two_Gibonacci_squares}, we have the sequence of equalities
\begin{align*}
    \GibSumSquared &= \gcd\left(G_{k+1}^2 - G_1^2,\; G_{k+2}^2 - G_2^2\right) &\text{by Theorem~\ref{thm:Simpler_GibSum_formula_for_k_even_versus_odd}}\\
    &= \gcd(F_k \cdot \gamma_{k,1},\; F_k \cdot \gamma_{k,2}) & \text{by Lemma~\ref{lem:difference_of_two_Gibonacci_squares}}\\
    &= F_k \cdot \gcd(\gamma_{k,1},\; \gamma_{k,2}).
\end{align*}
It suffices to show that $\gcd(\gamma_{k,1},\; \gamma_{k,2}) = \gcd(\mu,5)$. Setting $\beta_n \colonequals G_0^2 F_n + 2 G_0 G_1 F_{n+1} + G_1^2 F_{n+2}$, we see that $\gamma_{k,1} = \beta_k$ and $\gamma_{k,2} = \beta_{k+2}$. Moreover, the sequence $\left( \beta_n \right)_{n \geq 0}$ is itself a generalized Fibonacci sequence since the recursion $\beta_{n+2} = \beta_n + \beta_{n+1}$ holds for all $n \geq 0$ as follows:
\begin{align*}
    \beta_{n+2} &= G_0^2 F_{n+2} + 2 G_0 G_1 F_{n+3} + G_1^2 F_{n+4}\\
    &=G_0^2 (F_{n} + F_{n+1}) + 2 G_0 G_1 (F_{n+1} + F_{n+2}) + G_1^2 (F_{n+2} + F_{n+3})\\
    &= \left( G_0^2 F_n + 2 G_0 G_1 F_{n+1} + G_1^2 F_{n+2} \right) + \left( G_0^2 F_{n+1} + 2 G_0 G_1 F_{n+2} + G_1^2 F_{n+3} \right)\\
    &= \beta_n + \beta_{n+1}.
\end{align*}
Hence $\gcd(\beta_n, \beta_{n+1}) = \gcd(\beta_0, \beta_1)$ for all $n\geq 0$ by Lemma~\ref{lem:Consec_Gib_GCD}. Thus we have
$$ \gcd(\gamma_{k,1},\; \gamma_{k,2}) = \gcd(\beta_k, \beta_{k+2}) = \gcd(\beta_k, \beta_k + \beta_{k+1}) = \gcd(\beta_k, \beta_{k+1}) = \gcd(\beta_0, \beta_1).$$
Since $\beta_0 = 2 G_0 G_1 + G_1^2$ and $\beta_1 = G_0^2 + 2 G_0 G_1 + 2 G_1^2$, it follows that
\begin{align*}
    \gcd(\gamma_{k,1},\; \gamma_{k,2}) = \gcd(\beta_0, \beta_1) &= \gcd\left( 2 G_0 G_1 + G_1^2, \; G_0^2 + 2 G_0 G_1 + 2 G_1^2 \right)\\
    &= \gcd\left( 2 G_0 G_1 + G_1^2, \; G_0^2 + G_1^2 + (2 G_0 G_1 + G_1^2) \right)\\
    &= \gcd\left( 2 G_0 G_1 + G_1^2, \; G_0^2 + G_1^2 \right).
\end{align*}
We now show that $\gcd\left( 2 G_0 G_1 + G_1^2, \; G_0^2 + G_1^2 \right) = \gcd(\mu, 5)$ by proving the following two claims:
\begin{align}
    \gcd\left( 2 G_0 G_1 + G_1^2, \; G_0^2 + G_1^2 \right) &\geq \gcd(\mu, 5), \text{ and }\label{eq:inequality_1}\\
    \gcd\left( 2 G_0 G_1 + G_1^2, \; G_0^2 + G_1^2 \right) &\leq \gcd(\mu, 5).\label{eq:inequality_2}
\end{align}

To that end, suppose that for some prime $p$ and $j \geq 1$, we have $p^j$ divides $\gcd(\mu,5)$. Then $p^j$ divides 5, in particular, and hence $p=5$ and $j=1$, and so 5 divides $\mu$. Observe that
\begin{align}
    (2G_0 + G_1)^2 = 4 G_0^2 + 4 G_0 G_1 + G_1^2 &\equiv -G_0^2 - G_0 G_1 + G_1^2 \equiv \mu \pmodd{5}.\label{eq:mu_congruence}
\end{align}
But since 5 divides $\mu$, we deduce that $(2G_0 + G_1)^2 \equiv 0 \pmod{5}$, and hence 5 divides $2G_0 + G_1$. This yields the congruence $2 G_0 G_1 + G_1^2 = (2 G_0 + G_1) G_1 \equiv 0 \pmod{5}$, and so 5 divides $2 G_0 G_1 + G_1^2$. Moreover, observe that
\begin{align*}
    G_0^2 + G_1^2 = 5 G_0^2 - 4 G_0^2 + G_1^2 & \equiv - 4 G_0^2 + G_1^2 \pmodd{5}\\
    &\equiv - (4G_0^2 - G_1^2) \pmodd{5}\\
    &\equiv -(2 G_0 + G_1)(2 G_0 - G_1) \pmodd{5}\\
    &\equiv 0 \pmodd{5},
\end{align*}
since 5 divides $2G_0 + G_1$. Hence 5 divides $G_0^2 + G_1^2$. Thus $p^j$ divides both $2 G_0 G_1 + G_1^2$ and $G_0^2 + G_1^2$, and we conclude $p^j$ divides $\gcd\left( 2 G_0 G_1 + G_1^2, \; G_0^2 + G_1^2 \right)$, which proves Inequality~\eqref{eq:inequality_1} holds.

To prove Inequality~\eqref{eq:inequality_2} holds, suppose that  for some prime $p$ and $j \geq 1$, we have $p^j$ divides $\gcd\left( 2 G_0 G_1 + G_1^2, \; G_0^2 + G_1^2 \right)$. Then $p^j$ divides both $2 G_0 G_1 + G_1^2$ and $G_0^2 + G_1^2$. The prime $p$ cannot divide $G_1$ for otherwise it would also divide $G_0$, contradicting $\gcd(G_0,G_1)=1$ (recall Convention~\ref{conv:relatively_prime_initial_values}). Since $p^j$ divides $2 G_0 G_1 + G_1^2$ and $2 G_0 G_1 + G_1^2 = (2G_0+G_1)G_1$, then $p^j$ must divide $2 G_0 + G_1$ since $p$ does not divide $G_1$. Thus we have
\begin{align*}
    5G_0^2  &= \left( 4G_0^2 - G_1^2 \right) + \left(G_0^2 + G_1^2\right)\\
    &= (2G_0 + G_1) (2G_0 - G_1) + \left(G_0^2 + G_1^2\right)\\
    &\equiv 0 \pmodd{p^j},
\end{align*}
and hence $p^j$ divides $5G_0^2$. The prime $p$ cannot divide $G_0$ for otherwise it would also divide $G_1$, contradicting $\gcd(G_0,G_1)=1$. Thus, $p^j$ dividing $5G_0^2$ implies that $p^j$ divides 5, and hence $p=5$ and $j=1$. Since we proved earlier that $p^j$ divides $2G_0+G_1$, it follows that 5 divides $2G_0+G_1$. In particular, due to Congruence~\eqref{eq:mu_congruence}, we have 5 divides $\mu$. We conclude $p^j$ divides $\gcd(\mu,5)$, which proves Inequality~\eqref{eq:inequality_2} holds. Therefore  $\GibSumSquared = F_k \cdot \gcd(\mu,5)$, as desired.
\end{proof}

%%%%%%%%%%%%%%%%%%%%%%%%%%%%%%%
%%%%%%%%%%%%%%%%%%%%%%%%%%%%%%%

\subsection{Remarks on the closed form for \texorpdfstring{$\GibSumSquared$}{curly G squared} when \texorpdfstring{$k$}{k} is odd}

When $k$ is odd, Theorem~\ref{thm:Simpler_GibSum_formula_for_k_even_versus_odd} yields the following closed form for the value $\GibSumSquared$:
$$\GibSumSquared = \gcd(2\mu, g_k),$$
where $g_k \colonequals \gcd\left( G_{k+1}^2 - G_1^2, \; G_{k+2}^2 - G_2^2 \right)$. We immediately arrive at the following theorem as a consequence.

\begin{theorem}\label{thm:GibSumSquared_divides_mu_when_k_odd}
For all $k \geq 1$ with $k$ odd, we have $\GibSumSquared$ divides the value $|2\mu|$ where $\mu = G_1^2 - G_0 G_1 - G_0^2$.
\end{theorem}

\begin{proof}
Given any odd integer $k \geq 1$, Theorem~\ref{thm:Simpler_GibSum_formula_for_k_even_versus_odd} implies that $\GibSumSquared = \gcd(2\mu, g_k)$, where $g_k \colonequals \gcd\left( G_{k+1}^2 - G_1^2, \; G_{k+2}^2 - G_2^2 \right)$. Hence $\GibSumSquared$ divides $|2\mu|$, and the result follows.
\end{proof}

\begin{remark}
It is worthy to note that Theorem~\ref{thm:GibSumSquared_divides_mu_when_k_odd} does not necessarily hold for even $k$ values. For example, consider the Fibonacci setting with $k = 4$. Then $\FibSumSquared = F_4 = 3$ by Theorem~\ref{thm:FibSumSquared_closed_forms}. In the Fibonacci setting, $\mu = 1$, however 3 does not divide $|2 \mu|$, which equals 2.
\end{remark}

\begin{remark}\label{rem:all_divisors_of_2mu_attained}
It is also worthy to note that the sequence $(\GibSumSquared)_{k \geq 1}$ can attain all positive divisors of $|2 \mu|$ as values. For example, consider the Gibonacci sequence with initial values $G_0 = 3$ and $G_1 = 1$. Then $\mu = 1^2 - 1 \cdot 3 - 3^2 = -11$ and hence $|2\mu| = 22$. Using the closed form $\GibSumSquared = \gcd(2\mu, g_k)$ implied by Theorem~\ref{thm:Simpler_GibSum_formula_for_k_even_versus_odd}, we leave it to the reader to verify the following:
$$ \GibSumSquared = \begin{cases}
  1,  & \text{if $k=7$}, \\
  2,  & \text{if $k=3$}, \\
  11,  & \text{if $k=5$}, \\
  22,  & \text{if $k=15$}.
\end{cases} $$
\end{remark}

The following theorem provides a sufficiency condition for when the value $\GibSumSquared$ attains its maximal value $|2 \mu|$ in the case that $k$ is odd.

\begin{theorem}\label{thm:GibSumSquared_attains_max_2mu_value_when_k_odd}
Let $k \geq 1$ be odd, and set $\ell_k \colonequals \gcd(G_{k+1} - G_1, G_{k+2} - G_2)$. If $2\mu$ divides $\ell_k$, then $\mathcal{G}_{G_0, G_1}^2\!(k \ell) = |2\mu|$ for all odd integers $\ell \geq 1$.
\end{theorem}

\begin{proof}
Suppose that $k \geq 1$ is odd, and set $\ell_k \colonequals \gcd(G_{k+1} - G_1, G_{k+2} - G_2)$. Assume that $2\mu$ divides $\ell_k$. Then $2\mu$ divides both $G_{k+1} - G_1$ and $G_{k+2} - G_2$. By a simple inductive argument, the latter implies that $2\mu$ divides $G_{k+n} - G_n$ for all $n \geq 1$. Hence, it follows that $2\mu$ divides $G_{ki+1} - G_{k(i-1)+1}$ and also $G_{ki+2} - G_{k(i-1)+2}$ for all $1 \leq i \leq \ell$. Moreover, since $\sum_{i=1}^\ell \left(G_{ki+1} - G_{k(i-1)+1}\right) = G_{k\ell+1}-G_1$ and $\sum_{i=1}^\ell \left(G_{ki+2} - G_{k(i-1)+2}\right) = G_{k\ell+2}-G_2$, then we conclude that $2\mu$ divides both $G_{k\ell+1}-G_1$ and $G_{k\ell+2}-G_2$. Recalling that Theorem~\ref{thm:Simpler_GibSum_formula_for_k_even_versus_odd} implies the formula $\GibSumSquared = \gcd(2\mu, g_k)$, where $g_k \colonequals \gcd\left( G_{k+1}^2 - G_1^2, \; G_{k+2}^2 - G_2^2 \right)$, we have
\begin{align*}
    g_{k\ell} &= \gcd\left( G_{k\ell+1}^2 - G_1^2, \; G_{k\ell+2}^2 - G_2^2 \right)\\
    &= \gcd\bigl( (G_{k\ell+1} - G_1)(G_{k\ell+1} + G_1), \; (G_{k\ell+2} - G_2)(G_{k\ell+2} + G_2) \bigr),
\end{align*}
and hence $2\mu$ divides $g_{k\ell}$. Thus $\mathcal{G}_{G_0, G_1}^2\!(k \ell) = \gcd(2\mu, g_{k\ell}) = |2\mu|$ for all odd integers $\ell \geq 1$.
\end{proof}

\begin{example}
Consider the Gibonacci sequence $(G_n)_{n\geq0}$ with initial values $G_0 = 2$ and $G_1 = 7$. Then $\mu = 7^2 - 2 \cdot 7 - 2^2 = 31$. Also observe that for $k = 15$, we have
$$\ell_{15} = \gcd(G_{16} - G_1, G_{17} - G_2) = \gcd(8122, 13144) = 62,$$
and indeed $2\mu$ clearly divides $\ell_k$ since $2\mu = \ell_k$, in particular, in this case. Hence by Theorem~\ref{thm:GibSumSquared_attains_max_2mu_value_when_k_odd}, we know $\mathcal{G}_{G_0, G_1}^2\!(15 \ell) = 62$ for all odd integers $\ell \geq 1$. We leave it to the reader to verify this.
\end{example}

%%%%%%%%%%%%%%%%%%%%%%%%%%%%%%%%%%%%%
%%%%%%%%%%%%%%%%%%%%%%%%%%%%%%%%%%%%%
%%%%%%%%%    SECTION 5    %%%%%%%%%%%
%%%%%%%%%%%%%%%%%%%%%%%%%%%%%%%%%%%%%
%%%%%%%%%%%%%%%%%%%%%%%%%%%%%%%%%%%%%

\section{Closed forms in the Fibonacci and Lucas settings}\label{sec:closed_form_in_Fibonacci_and_Lucas_settings}

In this section, we prove the following closed forms on the GCD of the sum of $k$ consecutive squares of Fibonacci and Lucas numbers:
\begin{center}
    \begin{tabular}{|c||c|c|c|}
    \hline
         $k$ & $\FibSumSquared$ & $\LucSumSquared$ & Proof in this paper \\ \hline\hline
         $k$ even & $F_k$ & $5F_k$ & Theorems~\ref{thm:FibSumSquared_closed_forms} and \ref{thm:LucSumSquared_closed_forms}, respectively \\ \hline
         $k \equiv 3 \pmod{6}$ & 2 & 2 & Theorems~\ref{thm:FibSumSquared_closed_forms} and \ref{thm:LucSumSquared_closed_forms}, respectively \\ \hline
         $k \equiv 1,5 \pmod{6}$ & 1 & 1 & Theorems~\ref{thm:FibSumSquared_closed_forms} and \ref{thm:LucSumSquared_closed_forms}, respectively \\ \hline
    \end{tabular}
\end{center}

\begin{lemma}\label{lem:even_F_k_values}
For all $n \in \mathbb{Z}$, the value $F_n$ is even if and only if $3$ divides $n$.
\end{lemma}

\begin{proof}
This follows from the fact that $F_3=2$ and the well-known identity, $F_m$ divides $F_n$ if and only if $m$ divides $n$ (see Corollary~10.2 of Koshy~\cite[p.~173]{Koshy2018}).
\end{proof}

\begin{theorem}\label{thm:FibSumSquared_closed_forms}
For all $k \geq 1$, we have
$$ \FibSumSquared = \begin{cases}
  F_k,  & \text{if $k$ is even}, \\
  2, & \text{if $k \equiv 3 \pmodd{6}$},\\
  1, & \text{if $k \equiv 1,5 \pmodd{6}$}.
\end{cases} $$
\end{theorem}

\begin{proof}
Since $\mu = F_1^2 - F_0 F_1 - F_0^2 = 1$, Theorem~\ref{thm:closed_form_for_GibSumSquared_when_k_even} in the Fibonacci setting yields $\FibSumSquared = F_k$ when $k$ is even since 5 does not divide $\mu$. On the other hand, when $k$ is odd, Theorem~\ref{thm:GibSumSquared_divides_mu_when_k_odd} in the Fibonacci setting yields $\FibSumSquared$ divides $|2\mu|$, and hence $\FibSumSquared = 1 \text{ or } 2$. Recall that by Theorem~\ref{thm:Simple_GibSum_formula} in the Fibonacci setting, we have
$$\FibSumSquared = \gcd\left(F_k F_{k+1},\; F_{k+1}^2 - 1,\; F_{k+2}^2 - 1\right),$$
and hence if $\FibSumSquared = 2$, then 2 divides $F_k F_{k+1}$ so either $F_k$ or $F_{k+1}$ is even. However, 2 dividing $F_{k+1}^2 - 1$  implies $F_{k+1}$ is odd, and thus $F_k$ must be even. By Lemma~\ref{lem:even_F_k_values}, it follows that 3 divides $k$, and in particular $k \equiv 3 \pmod{6}$ since $k$ is odd. On the other hand, if $k \equiv 3 \pmod{6}$, then 3 divides $k$ and hence $F_k$ is even by Lemma~\ref{lem:even_F_k_values}, and so both $F_{k+1}$ and $F_{k+2}$ are both odd implying $F_{k+1}^2 - 1$ and $F_{k+2}^2 - 1$ are both even. Thus $\FibSumSquared$ is even; that is, $\FibSumSquared = 2$ is forced. Thus if $k$ is odd, then $\FibSumSquared = 2$ if and only if $k \equiv 3 \pmod{6}$. And consequently if $k$ is odd, then $\FibSumSquared = 1$ if and only if $k \equiv 1,5 \pmod{6}$.
\end{proof}

To prove the Lucas version of Theorem~\ref{thm:FibSumSquared_closed_forms}, we first give the Lucas version of the necessary and sufficient condition for $L_n$ to be even (compare this with the Fibonacci version given in Lemma~\ref{lem:even_F_k_values}).

\begin{lemma}\label{lem:even_L_k_values}
For all $n \in \mathbb{Z}$, the value $L_n$ is even if and only if $3$ divides $n$.
\end{lemma}

\begin{proof}
See Identity~(23.2) in Theorem~23.1 of Koshy~\cite[p.~462]{Koshy2018}.
\end{proof}

\begin{theorem}\label{thm:LucSumSquared_closed_forms}
For all $k \geq 1$, we have
$$ \LucSumSquared = \begin{cases}
  5F_k,  & \text{if $k$ is even}, \\
  2, & \text{if $k \equiv 3 \pmodd{6}$},\\
  1, & \text{if $k \equiv 1,5 \pmodd{6}$}.
\end{cases} $$
\end{theorem}

\begin{proof}
Since $\mu = L_1^2 - L_0 L_1 - L_0^2 = -5$, Theorem~\ref{thm:closed_form_for_GibSumSquared_when_k_even} in the Lucas setting yields $\LucSumSquared = 5 F_k$ when $k$ is even since 5 divides $\mu$. On the other hand, when $k$ is odd, Theorem~\ref{thm:GibSumSquared_divides_mu_when_k_odd} in the Lucas setting yields $\LucSumSquared$ divides $|2 \mu|$, and hence $\LucSumSquared = 1, 2, 5, \text{ or } 10$. Recall that by Theorem~\ref{thm:Simple_GibSum_formula} in the Lucas setting, we have
$$\LucSumSquared = \gcd\left(L_k L_{k+1} - 2,\; L_{k+1}^2 - 1,\; L_{k+2}^2 - 9\right).$$
We will first rule out the possibility that $\LucSumSquared = 5 \text{ or } 10$. Suppose by way of contradiction that 5 divides $\LucSumSquared$. Then 5 divides both $L_{k+1}^2 - 1$ and $L_{k+2}^2 - 9$, and hence we have the congruences
\begin{align}
    L_{k+1}^2 &\equiv 1 \pmodd{5} \label{cong:1}\\
    L_{k+2}^2 &\equiv 4 \pmodd{5}. \label{cong:2}
\end{align}
It is readily verified that $\pi_L(5) = 4$, where $\pi_L(5)$ is the Pisano period of the Lucas sequence modulo~5 (recall Definition~\ref{def:Gisano_period}). This length 4 period repeats the sequence terms $(L_n \pmod{5})_{n=0}^3 = (2, 1, 3, 4)$. Squaring this sequence we get $(L_n^2 \pmod{5})_{n=0}^3 = (4,1,4,1)$, a repeating sequence of length 2. More precisely, $L_n^2 \equiv 4 \pmod{5}$ if and only if $n$ is even, and hence Congruences~\eqref{cong:1} and \eqref{cong:2} force $k$ to be even, which is a contradiction. Therefore, $\LucSumSquared \not= 5 \text{ or } 10$, and thus $\LucSumSquared = 1 \text{ or } 2$.

If $\LucSumSquared = 2$, then 2 divides $L_k L_{k+1} - 2$, $L_{k+1}^2 - 1$, and $L_{k+2}^2 - 9$. Since 2 divides $L_k L_{k+1} - 2$, then either $L_k$ or $L_{k+1}$ is even. However, 2 dividing $L_{k+1}^2 - 1$  implies $L_{k+1}$ is odd, and thus $L_k$ must be even. By Lemma~\ref{lem:even_L_k_values}, it follows that 3 divides $k$, and in particular $k \equiv 3 \pmod{6}$ since $k$ is odd. On the other hand, if $k \equiv 3 \pmod{6}$, then 3 divides $k$ and hence $F_k$ is even by Lemma~\ref{lem:even_L_k_values}, and so $L_{k+1}$ and $L_{k+2}$ are both odd implying $L_{k+1}^2 - 1$ and $L_{k+2}^2 - 1$ are both even. Thus $\LucSumSquared$ is even; that is, $\LucSumSquared = 2$ is forced. Thus if $k$ is odd, then $\LucSumSquared = 2$ if and only if $k \equiv 3 \pmod{6}$. And consequently if $k$ is odd, then $\LucSumSquared = 1$ if and only if $k \equiv 1,5 \pmod{6}$.
\end{proof}

%%%%%%%%%%%%%%%%%%%%%%%%%%%%%%%%%%%%%
%%%%%%%%%%%%%%%%%%%%%%%%%%%%%%%%%%%%%
%%%%%%%%  OPEN QUESTIONS  %%%%%%%%%%%
%%%%%%%%%%%%%%%%%%%%%%%%%%%%%%%%%%%%%
%%%%%%%%%%%%%%%%%%%%%%%%%%%%%%%%%%%%%

\section{Open questions}\label{sec:open questions}

\subsection{Extension to higher powers}

\begin{question}
For $k$ even, the formulas for the GCD of all sums of $k$ consecutive Gibonacci numbers and the GCD of all sums of $k$ consecutive squares of Gibonacci numbers, respectively are
\begin{align*}
    \GibSum &= \gcd(G_{k+1}-G_1,\, G_{k+2}-G_2)\\
    \GibSumSquared &= \gcd\left(G_{k+1}^2 - G_1^2,\; G_{k+2}^2 - G_2^2\right).
\end{align*}
The first formula holds from Guyer-Mbirika~\cite[Theorem~15]{Guyer_Mbirika2021}, while the second formula holds from Theorem~\ref{thm:Simpler_GibSum_formula_for_k_even_versus_odd} in this current paper. For certain even values $k$ and $n \geq 3$, it seems reasonable that the following may hold:
$$ \mathcal{G}_{G_0, G_1}^n\!(k) = \gcd\left(G_{k+1}^n - G_1^n,\; G_{k+2}^n - G_2^n\right). $$
Although this does not appear to be true even in the case of $n=3$, data collected via \texttt{Mathematica} for this $n$ value in the Fibonacci and Lucas settings supports the following conjectures for $k$ even:
\begin{align*}
    \mathcal{F}^3(k) &= \begin{cases}
  \gcd\left(F_{k+1}^3 - 1,\; F_{k+2}^3 - 1\right),  & \text{if $6$ divides $k$}, \\
  \frac{1}{2} \cdot \gcd\left(F_{k+1}^3 - 1,\; F_{k+2}^3 - 1\right),  & \text{if $6$ does not divide $k$}.
  \end{cases}\\
  \mathcal{L}^3(k) &= \begin{cases}
  \gcd\left(L_{k+1}^3 - 1,\; L_{k+2}^3 - 9\right),  & \text{if $6$ divides $k$}, \\
  \frac{1}{2} \cdot \gcd\left(L_{k+1}^3 - 1,\; L_{k+2}^3 - 9\right),  & \text{if $6$ does not divide $k$}.
\end{cases}
\end{align*}
Can this conjecture not only be proved, but also extended to higher values $n \geq 4$?
\end{question}

\begin{question}
Fix $G_0, G_1 \in \mathbb{Z}$ and set $d := \gcd(G_0,G_1)$. Consider the two Gibonacci sequences $(G_n)_{n\geq 0}$ and $(G_n^\prime)_{n\geq 0}$, where $(G_n^\prime)_{n=0}^\infty$ is generated by the relatively prime initial conditions $G_0^\prime = \frac{G_0}{d}$ and $G_1^\prime = \frac{G_1}{d}$. Then the following identities holds:
\begin{align*}
    \GibSum &= d \cdot \GibSumPrime\\
    \GibSumSquared &= d^2 \cdot \GibSumPrimeSquared.
\end{align*}
The first formula holds from Guyer-Mbirika~\cite[Theorem~19]{Guyer_Mbirika2021}, while the second formula holds from Theorem~\ref{thm:relatively_prime_initial_values_only} in this current paper. Will it be the case that $\mathcal{G}_{G_0, G_1}^n\!(k) = d^n \cdot \mathcal{G}_{G_0^\prime, G_1^\prime}^n\!(k)$ for all $n \geq 3$?
\end{question}

%%%%%%%%%%%%%%%%%%%%%%%%%%%%%%%
%%%%%%%%%%%%%%%%%%%%%%%%%%%%%%%

\subsection{Periodicity of the sequences \texorpdfstring{$(\GibSum)_{k \geq 1}$}{GibSum} and \texorpdfstring{$(\GibSumSquared)_{k \geq 1}$}{GibSumSquared} when \texorpdfstring{$k$}{k} is odd}

\begin{example}
Consider the Gibonacci sequence with initial values $G_0 = -1$ and $G_1 = 3$. Then $\mu = 3^2 - (-1)(3) - (-1)^2 = 11$. Moreover for $k$ odd, we have the following (formal proofs omitted, but easily verified using the formulas for $\GibSum$ in Guyer-Mbirika~\cite[Theorem~15]{Guyer_Mbirika2021} and $\GibSumSquared$ in Theorem~\ref{thm:Simple_GibSum_formula} in this current paper, respectively):
$$ \GibSum = \begin{cases}
  2, & \text{if $k \equiv 3 \pmodd{6}$},\\
  1, & \text{if $k \equiv 1,5 \pmodd{6}$},
\end{cases}
$$
and
$$\GibSumSquared = \begin{cases}
  22, & \text{if $k \equiv 15 \pmodd{30}$},\\
  11, & \text{if $k \equiv 5,25 \pmodd{30}$},\\
  2, & \text{if $k \equiv 3,9,21,27 \pmodd{30}$},\\
  1, & \text{if $k \equiv 1,7,11,13,17,19,23,29 \pmodd{30}$}.
\end{cases}$$
Hence we have the periodic relationships in the sequences $(\GibSum)_{k \geq 1}$ and $(\GibSumSquared)_{k \geq 1}$:
$$ \mathcal{G}_{G_0, G_1}\!(k+6) = \GibSum \;\text{ and }\; \mathcal{G}_{G_0, G_1}^2\!(k+30) = \GibSumSquared,$$
for all odd $k$.
\end{example}

\begin{question}
Under what conditions on the initial values $G_0$ and $G_1$ are $\GibSum$ and $\GibSumSquared$ periodic on odd $k$ values?
\end{question}

%%%%%%%%%%%%%%%%%%%%%%%%%%%%%%%%%%%%%
%%%%%%%%%%%%%%%%%%%%%%%%%%%%%%%%%%%%%
%%%%%%%%% ACKNOWLEDGMENTS  %%%%%%%%%%
%%%%%%%%%%%%%%%%%%%%%%%%%%%%%%%%%%%%%
%%%%%%%%%%%%%%%%%%%%%%%%%%%%%%%%%%%%%

\section*{Acknowledgments}
First author Mbirika gratefully thanks second author Spilker for his close reading of an earlier publication by Guyer and Mbirika, which led to this current collaboration between Mbirika and Spilker. We also thank \texttt{Mathematica} for much computer evidence that helped us conjecture some of the GCD values that led to the main results in this current paper. Finally, we thank the anonymous referee for helpful comments which improved the exposition in this paper.

%%%%%%%%%%%%%%%%%%%%%%%%%%%%%%%%%%%%%
%%%%%%%%%%%%%%%%%%%%%%%%%%%%%%%%%%%%%
%%%%%%%%%   BIBLIOGRAPHY  %%%%%%%%%%%
%%%%%%%%%%%%%%%%%%%%%%%%%%%%%%%%%%%%%
%%%%%%%%%%%%%%%%%%%%%%%%%%%%%%%%%%%%%

\medskip

\noindent 2010 {\it Mathematics Subject Classification}:
Primary 11B39, Secondary 11A05, 11B50.


\begin{thebibliography}{10}


\bibitem{Benjamin2003} A.~Benjamin and J.~Quinn, \emph{Proofs that Really Count: The Art of Combinatorial Proof}, Mathematical Association of America, 2003.

\bibitem{Cerin2013} Z.~\v{C}erin, On factors of sums of consecutive Fibonacci and Lucas numbers, \emph{Ann. Math. Inform.}, \textbf{41} (2013), 19--25.

\bibitem{Demirturk2010} B.~Demirt\"urk, Fibonacci and Lucas sums by matrix methods, \emph{Int. Math. Forum}, \textbf{5} (2010), 99--107.


\bibitem{Guyer_Mbirika2021}
D.~Guyer and A.~Mbirika, GCD of sums of $k$ consecutive Fibonacci, Lucas, and generalized Fibonacci numbers,  \textit{J. Integer Seq.}, \textbf{24} (2021), no.~9, Article~21.9.8, 25~pp.

\bibitem{Iyer1969}
M.~Iyer, Sums involving {F}ibonacci numbers, \emph{Fibonacci Quarterly}, \textbf{7.1} (1969), 92--98.

\bibitem{Koshy2018}
T.~Koshy, \emph{{F}ibonacci and {L}ucas Numbers with Applications}, Volume I, second edition, Wiley, 2018.

\bibitem{Ruggles1963} I.~D.~Ruggles, Elementary problem B-1, \emph{Fibonacci Quarterly}, \textbf{1.1} (1963), 73.

\bibitem{Shtefan2018} D.~Shtefan and I.~Dobrovolska, The sums of the consecutive Fibonacci numbers, \emph{Fibonacci Quarterly}, \textbf{56.3} (2018), 229--236.

\bibitem{Tagiuri1901} A.~Tagiuri, Di alcune successioni ricorrenti a termini interi e positivi, \emph{Periodico di Matematica}, \textbf{16} (1901), 1--12.

\bibitem{Tekcan2007} A.~Tekcan, B.~Gezer, and O.~Bizim, Some relations on Lucas numbers and their sums, \emph{Adv. Stud. Contemp. Math. (Kyungshang)}, \textbf{15} (2007), 195--211.

\bibitem{Tekcan2008} A.~Tekcan, A.~\"{O}zko\c{c}, B.~Gezer, and O.~Bizim, Some relations involving the sums of Fibonacci numbers, \emph{Proc. Jangjeon Math. Soc.}, \textbf{11} (2008), 1--12.

\bibitem{Vajda1989}
S.~Vajda, \emph{{F}ibonacci and {L}ucas Numbers, and the Golden Section}, Ellis Horwood Limited, 1989.


\end{thebibliography}
\end{document}